\documentclass{ijuc}
\usepackage[pdftex]{graphicx}
\usepackage{courier}
\usepackage{amsmath}
\usepackage{url}
\usepackage{amsthm}
\usepackage{subfig}
\usepackage{dsfont}
\usepackage{enumerate}

\newtheorem*{r1}{Rule 1}
\newtheorem*{r2}{Rule 2}
\newtheorem*{r3}{Rule 3}
\newtheorem{thm}{Theorem}

\begin{document}

\title{Cellular Automata to More Efficiently Compute the Collatz Map}
\author{Sitan Chen\email{sitanchen@college.harvard.edu}
}

\institute{Department of Mathematics, Harvard College, Cambridge, MA 02138}

\def\received{Received 13 January 2013; In final form 13 January 2013}

\maketitle
\begin{abstract}
The Collatz, or $3x+1$, Conjecture claims that for every positive integer $n$, there exists some $k$ such that $T^k(n)=1$, where $T$ is the Collatz map. We present three cellular automata (CA) that transform the global problem of mimicking the Collatz map in bases 2, 3, and 4 into a local one of transforming the digits of iterates. The CAs streamline computation first by bypassing calculation of certain parts of trajectories: the binary CA bypasses division by two altogether. In addition, they allow for multiple trajectories to be calculated simultaneously, representing both a significant improvement upon existing sequential methods of computing the Collatz map and a demonstration of the efficacy of using a massively parallel approach with cellular automata to tackle iterative problems like the Collatz Conjecture. \\\\
\emph{Keyphrases}: Collatz Conjecture, massively parallel, deterministic computational model, cellular automata 
\end{abstract}

\section{INTRODUCTION}

Consider the \emph{Collatz map} $T:\mathds{N}\to\mathds{N}$ defined by $$T(n)=\begin{cases} 3n+1, & n\equiv1\mod{2} \\ n/2, & n\equiv0\mod{2}.\end{cases}$$  Given $n$, define the \emph{total stopping time} to be the smallest $k$ for which $T^k(n)=1$. The \emph{Collatz Conjecture} claims that every $n$ has a finite total stopping time; it has been verified for inputs as high as $20\cdot2^{58}$ by Oliveira e Silva [1], and a generalization of the Conjecture by Kurtz and Simon was shown to be recursively undecidable [2].  If we define the \emph{stopping time} to be the smallest $k$ for which $T^k(n)<n$, the claim that every $n$ has a finite stopping time is equivalent to the Collatz Conjecture. 

The \emph{trajectory} for an input $n$ is the sequence $$n, T(n), T^2(n), T^3(n), ...$$  We can categorize the possible behaviors of the trajectory as follows [3]:

\begin{enumerate}[(i)]
\item\emph{Convergent trajectory}: $T^k(n)=1$ for some positive integer $k$.

\item\emph{Non-trivial cyclic trajectory}: $T^k(n)$ becomes periodic and $T^k(n)\neq1$ for any positive $k$.

\item\emph{Divergent trajectory}: $\displaystyle\lim_{k\to\infty}T^k(n)=\infty$.
\end{enumerate}

The iterative nature of the problem and the complexity of behavior that arises from the simple premises of the Collatz Conjecture suggest the use of cellular automata to mimic computation.

Notable instances of using deterministic computational models to simulate the behavior of the Collatz map include De Mol's 2-tag system using a set of three production rules over an alphabet of three symbols to mimic the Collatz map [4], Michel's two one-tape Turing machines [5] whose halting problems depend on generalizations of the Collatz problem, and Bruschi's two one-dimensional cellular automata that calculate trajectories using \emph{tagging} operations to detect parity [6]. In this paper, we demonstrate one three-dimensional and two two-dimensional cellular automata that also mimic the behavior of the Collatz map but whose evolution laws operate independent of parity, depending solely on local relationships between digits of the iterates. We find that our automata have the potential to compute trajectories more efficiently than do current methods both by bypassing computation of large parts of trajectories and by verifying an arbitrary number of inputs in parallel.

\section{CA1- BASE 3}

This CA acts upon an three-dimensional grid of cubical cells in two layers indexed by 0 and 1, respectively. Each cell in the bottom layer is in one of four states represented in Figure~\ref{fig:basethreerules} by the corresponding colors in parentheses-- \texttt{0} (light gray), \texttt{1} (gray), \texttt{dark gray}, and \texttt{empty} (white)--- and each cell in the top layer is in one of four states also represented in Figure~\ref{fig:basethreerules} by the corresponding colors in parentheses-- \texttt{odd-normal} (gray), \texttt{odd-special} (odd-normal with a cross), \texttt{even} (light gray), and \texttt{unknown-parity} (black).  In the planes of each layer, rows are indexed top to bottom by $\mathds{N}\cup\{0\}$ and columns right to left by $\mathds{Z}$, where row $i$ and column $j$ of the bottom layer lie respectively beneath row $i$ and column $j$ of the top layer. We denote a cell in row $i$ and column $j$ of layer $k$ by the ordered triple $(i,j,k)$ and its state at time period $t$ by $s_{i,j,k}(t)$.

As shown in Figure~\ref{fig:ca3bottom}, given a cell in the bottom row (dark gray), the neighborhood of that cell is the set of cells labeled B-F (light gray), where cell B lies directly under cell F. As shown in Figure~\ref{fig:ca3top}, given a cell in the bottom row (dark gray), the neighborhood of that cell is the set of cells labeled D, G, H (dark gray). In other words, the neighborhood of cell $(i,j,0)$ is the set of cells $\{(i-1,j,0),(i-1,j,1),(i-1,j-1,0),(i-1,j-1,1),(i,j-1,0)\}$, while the neighborhood of cell $(i,j,1)$ is the set of cells $\{(i,j,0),(i,j-1,1)\}$.

\begin{figure}[h]
 \begin{center}
 \subfloat[Bottom layer]{\label{fig:ca3bottom}{\includegraphics[height=3cm]{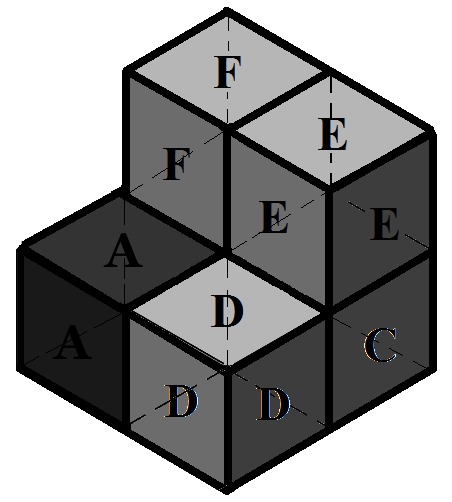}}}
\hspace{0.3cm}
 \subfloat[Top layer]{\label{fig:ca3top}\includegraphics[height=3cm]{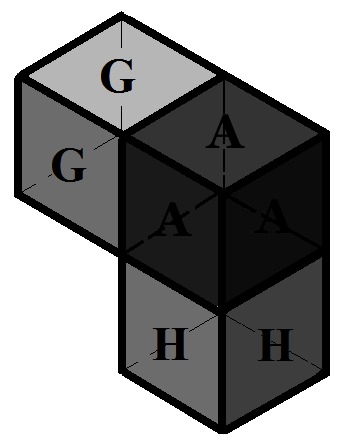}}
 \caption{Neighborhoods for cells in different layers}
 \end{center}
\end{figure}

If input $n$ has ternary representation $\sum^{N}_{i=0}3^id_i$, the CA is initialized as follows: for $i$ from $0$ to $N$ and an arbitrary $k$, we have $s_{0,k+i,0}(0)=d_i$. All cells in the top layer are set to state \texttt{unknown-parity}, and all other cells in the bottom layer are set to state \texttt{empty}.

\begin{figure}[h]
  \centering
      \includegraphics[width=\textwidth]{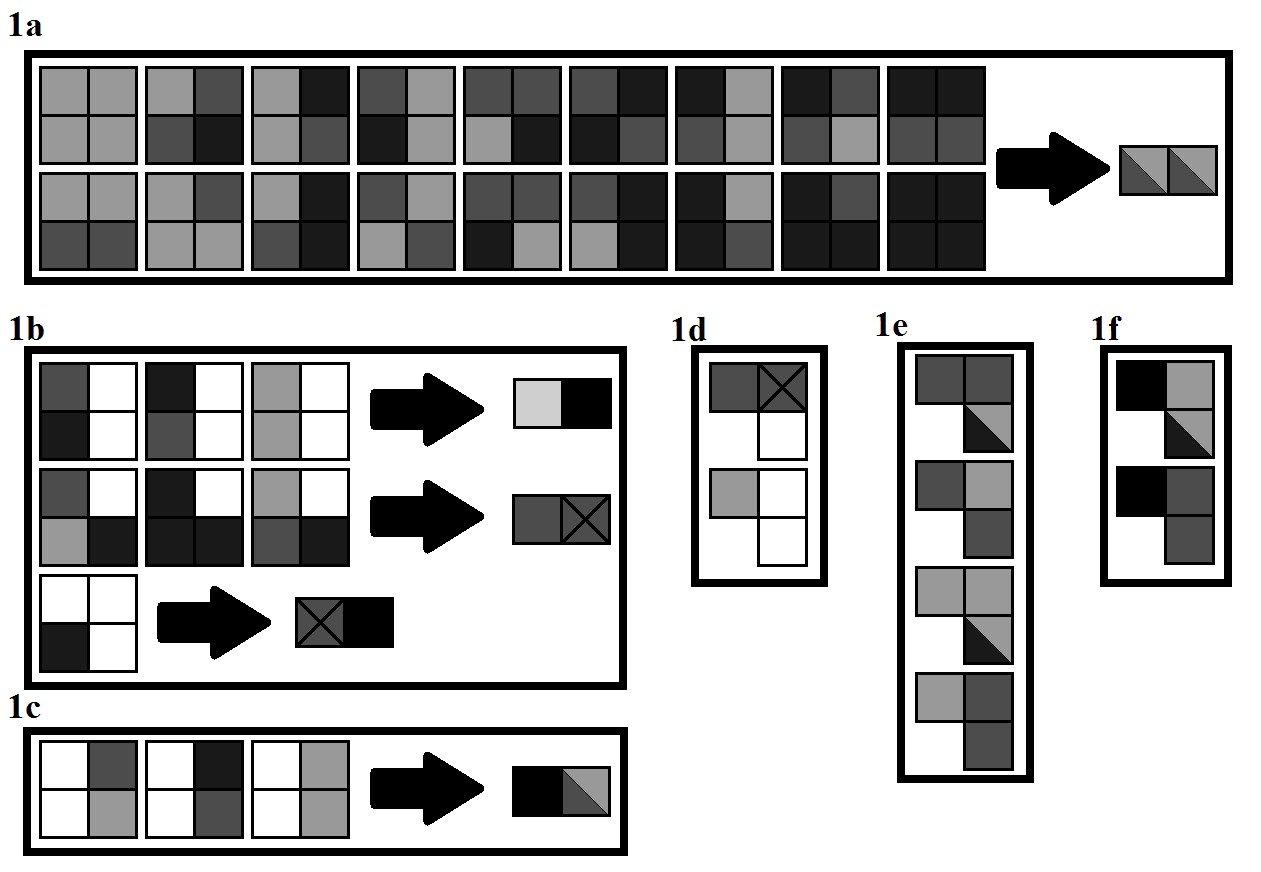}
  \caption{CA1 Evolutions Laws}
  \label{fig:basethreerules}
\end{figure}

For simplicity, the rules in Figure~\ref{fig:basethreerules} are presented in two dimensions: rules of categories 1a through 1c represent transformation of the bottom layer, while rules of categories 1d through 1g represent transformation of the top layer.

A rule in category 1a, 1b, or 1c is represented as a $2\times2$ grid with an arrow directed at a $2\times1$ grid (for convenience, the $2\times2$ grids of rules in Figure~\ref{fig:basethreerules} that involve the same $2\times1$ grid are grouped together). They correspond, respectively, to the bottom and top halves of the configuration in Figure~\ref{fig:ca3bottom}. Specifically, in the $2\times 2$ grid, the bottom-left square corresponds to $A$ in Figure~\ref{fig:ca3bottom}, the cell that is evolving, and its color denotes the future state into which that cell evolves given the states in its neighborhood. If this cell has coordinates $(i,j,0)$, proceeding clockwise, the grid's other cells correspond to the current states of cells B$(i-1,j,0)$, C$(i-1,j-1,0)$, and D$(i,j-1,0)$ respectively. In the $2\times 1$ grid, the left and right cells correspond to cells $F(i-1,j,1)$ and $E(i-1,j-1,1)$, respectively. For example, the first rule in category 1c of Figure~\ref{fig:basethreerules} should be interpreted to mean that $s_{i,j,0}(t+1)=\texttt{empty}$ if $s_{i-1,j,0}(t)=\texttt{empty}$, $s_{i-1,j,1}(t)=\texttt{unknown-parity}$, $s_{i-1,j-1,1}(t)=\texttt{odd-normal} \ \text{or} \ \texttt{even}$, $s_{i-1,j-1,0}(t)=\texttt{1}$, and $s_{i,j-1,0}(t)=\texttt{0}$.

A rule in category 1d, 1e, or 1f is represented as an L-shaped collection of three cells, corresponding to the configuration in Figure~\ref{fig:ca3top}.  Specifically, the upper-right square in each rule corresponds to cell $A$ in Figure~\ref{fig:ca3top}, that is, the cell that is evolving, and the color of the square denotes the future state into which that cell evolves given the state in its neighborhood. If this cell has coordinates $(i,j,1)$, the cells to its left, right, and bottom correspond to current states of cells $G(i+1,j,1)$ and $H(i,j-1,1)$, respectively. For example, the first rule in category 1e of Figure~\ref{fig:basethreerules} should be interpreted to mean that $s_{i,j,1}(t+1)=\texttt{odd-normal}$ if $s_{i,j+1,1}(t)=\texttt{odd-normal}$ and $s_{i,j,0}(t)=\texttt{0} \ or \ \texttt{2}$.

\begin{r1}If a cell's neighborhood corresponds to that of any of the evolution laws (Figure~\ref{fig:basethreerules}), then its state becomes the state of the cell in that evolution law corresponding to cell $A$.\end{r1}

\begin{r2}If the neighborhood of a cell in the top layer does not correspond to any evolution law, then that cell's state becomes \texttt{unknown-parity}.\end{r2}

\begin{r3}If the neighborhood of a cell in the bottom layer does not correspond to any evolution law, then that cell's state becomes \texttt{empty}.\end{r3}

\begin{figure}[h]
  \centering
      \includegraphics[height=1.5in]{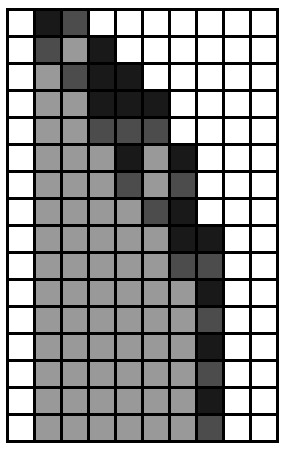}
  \caption{CA1 in action for $T^{0}=7$}
  \label{fig:ca1inaction}
\end{figure}

In Figure~\ref{fig:ca1inaction}, if we treat the non-empty cells in each row as ternary digits, the CA gives a sequence of iterates: $$7, 11, 17, 26, 13, 20, 10, 5, 8, 4, 2, 1, 2, 1, 2, 1.$$ As we will prove, CA1 thus mimics the following modified Collatz map $T_1:\mathds{N}\to\mathds{N}$ defined by \begin{equation*}T_1(n)=
\begin{cases}
(3n+1)/2, & n\equiv1\mod{2} \\
n/2, & n\equiv0\mod{2}.
\end{cases}
\end{equation*}
\begin{thm}
For every $m$, there exists some $i$ so that after $i$ iterations of CA1, the states of cells not in the state of \texttt{empty} in row $k$ of the bottom layer correspond to the ternary digits of $T^k_1(n)$, where $n$ is the input, for all $0\le k\le m$.
\end{thm}

\begin{proof}
In base three, parity of a number is the same as the parity of the sum of that number's digits. So whereas the bottom layer of the CA represents transformations of the iterates, the upper layer serves to add digits $\pmod 2$ to determine parity. As we shall show, the latter goes about doing this by computing successive partial sums, starting with the most significant digits.

In other words, given an iterate $\sum^{N-1}_{k=0}3^kd_k$, after $i$ generations have elapsed for the CA, the cell above the cell representing $d_{N-i+1}$ will have state equal to the parity of $\sum^{N}_{k=N-i+1}d_k$, so after $N+1$ generations, the state of the cell above $d_{0}$ will be the parity of the iterate.

We claim that the rules in categories 1d, 1e, and 1f are exhaustive and properly determine the parity of an iterate in this way. We will denote cell $(i,j)$ to be the cell that is evolving.

If $s_{i,j+1,1}(t)=\texttt{unknown-parity}$, the CA is determining the parity of the leading digit $(i,j,0)$, so $s_{i,j,1}(t+1)$ is simply the parity of $s_{i,j,0}(t)$, which can take on three possible values--\texttt{0}, \texttt{1}, \texttt{2}--giving the three evolution laws in category 1f.

Otherwise, if $s_{i,j,0}(t)\neq\texttt{empty}$, the CA is calculating the parities of partial sums of the $k$ significant digits, and there are three possible states for $(i,j,0)$--\texttt{0}, \texttt{1}, \texttt{2}-- and two for $(i,j+1,1)$--\texttt{odd-normal}, \texttt{even}--making for the six evolution laws in categories 1e. In each case, $s_{i,j,1}(t+1)=s_{i,j-1,1}(t)+s_{i,j,0}(t)$ in the sense of adding parities so that the sum of state \texttt{even} and \texttt{1} is \texttt{odd-normal} and the sum of \texttt{odd} and \texttt{1} is \texttt{even}, for example.

If $s_{i,j,0}(t)=\texttt{empty}$, the CA has reached the right end of the iterate. If the iterate is odd and the coordinates of the units digit are $(x,y,0)$, the CA sets the state of $(x,y-1,1)$ to \texttt{odd-special}
to signal the CA to append a 1 to the iterate (see below), giving the first rule in category 1d. The last two rules in this category tell the CA to leave all top-layer cells in row $i$ to the right of $(i,j-1,1)$ (for odd iterates, to the right of $(i,j,1)$) in state \texttt{unknown-parity}.

Therefore, the rules in categories 1d to 1g properly determine the parity of iterates.

It suffices to show that the rules in categories 1a to 1c are exhaustive and properly mimic the Collatz map in the bottom layer. In particular, we will show that if $T^k(n)=\sum^{N}_{i=0}3^id_i$ is even and cell $(x,y+i,0)$ has state $d_i$ for $0\le i\le N$, then if $T^{k+1}(n)=\sum^{N'}_{i=0}3^id'_i$, cell $(x+1,y+i,0)$ will eventually have state $d'_i$ for $0\le i\le N'$. If $T^k(n)$ is odd and if $T^{k+1}(n)=\sum^{N'}_{i=0}3^id'_i$, cell $(x+1,y+i-1,0)$ has state $d'_i$ for $0\le i\le N'$.

In our proof, we will denote the cell in the bottom-left of each evolution law in Figure~\ref{fig:basethreerules}, that is, the cell that is evolving, by $(i,j,0)$. We will continue to denote the current iterate by $\sum^{N'}_{k=0}3^kd'_k$ and the iterate in the preceding row by $\sum^{N}_{k=0}3^kd_k$.

If $s_{i-1,j-1,0}(t)=\texttt{empty}$, the CA is determining the rightmost digits of the iterate on row $i$. The first and last rows of category 1b in Figure~\ref{fig:basethreerules} represent the two possible cases for $(s_{i-1,j,1}(t),s_{i-1,j-1,1}(t))$ for when the previous iterate is even or odd, respectively. In the former case, the CA should mimic division by two so that $2d'_0=2s_{i,j,0}(t+1)\equiv s_{i-1,j,0}(t)=d_0\pmod{3}$. In the latter case, the CA should mimic $T(n)=(3n+1)/2$. When $s_{i-1,j,1}(t)=\texttt{odd-special}$ as in the last rule of category 1b, the CA effectively treats the empty cell $(i-1,j,0)$ as being in state \texttt{1} so that $2d'_0=2s_{i,j,0}(t+1)\equiv 1\pmod{3}$, meaning the state of $(i,j,0)$, the units digit of the new iterate will be \texttt{2}. The digit $d'_1$ must then satisfy $2(3d'_1+d'_0)\equiv 3d_0+1\pmod{9}$, giving the second row of rules in category 1b.

If $s_{i-1,j,0}(t)=\texttt{empty}$, the CA is determining the leftmost digits of the iterate on row $i$. By the above discussion of the process by which the CA determines parity, cell $(i-1,j,1)$ must be in state \texttt{unknown-parity}. Furthermore, $s_{i,j-1,0}(t)$ cannot exceed $s_{i-1,j-1,0}(t)$. Otherwise, if we denote $x$ to be the iterate in row $i-1$ and $y$ the iterate in row $i$, then if $\sum^{N}_{k=0}3^kd_k$ is even, $x/2<y$, and if it is odd, $(3x+1)/2<y$, both contradictions! In fact, for there to be no such contradictions when $s_{i-1,j-1,0}(t)$ is \texttt{1} \ or \ \texttt{2}, $s_{i,j-1,0}(t)$ must be strictly less than $s_{i-1,j-1,0}(t)$. This gives exactly one possibility for each choice of $s_{i-1,j-1,0}(t)$ from $\{\texttt{0},\texttt{1},\texttt{2}\}$, and in each case, it is clear that because $(i,j-1,0)$ represents the leftmost digit, $s_{i,j,0}(t+1)=\texttt{empty}$, giving the rules in category 1c.

If none of the cells in the neighborhood of $(i,j)$ are  \texttt{empty}, the CA is determining the inner digits of the iterate, so $s_{i-1,j,1}(t)$ and $s_{i-1,j-1,1}(t)$ must be \texttt{even} or \texttt{odd-normal} as in the rules in category 1a. First note that multiplication by two in base three involves carries of at most one. Therefore, for every pair $(s_{i-1,j,0}(t),s_{i-1,j-1,0}(t))\in\{\texttt{0,1,2}\}\times\{\texttt{0,1,2}\}$, $2s_{i,j-1,0}(t)\equiv s_{i-1,j-1,0}(t)+\epsilon\pmod{3}$, where $\epsilon\in\{0,1\}$, giving the 18 possible neighborhoods in category 1a. In each of these cases, $s_{i,j,0}(t+1)$ satisfies $$2\left(3s_{i,j,0}(t+1)+s_{i,j-1,0}(t)\right)+\epsilon\equiv3s_{i-1,j,0}(t)+s_{i-1,j-1,0}(t)\pmod{9}$$ giving the evolution laws in category 1a.\end{proof}

\section{CA2 -- BASE 4}

This cellular automaton acts on a two-dimensional infinite grid of square cells each in one of five possible states ---represented in Figure~\ref{fig:basefourlaws} by the corresponding colors in parentheses---\texttt{0} (white), \texttt{1} (light gray), \texttt{2} (dark gray), \texttt{3} (black), and \texttt{empty} (white). Additionally, states \texttt{0} through \texttt{3} each have an additional attribute of either \texttt{-odd} or \texttt{-even}, indicated by a circle or a square within the cell, respectively (e.g. a state of \texttt{0-odd} is a white cell containing a circle).

Let the rows be indexed top to bottom by $\mathds{N}\cup\{0\}$ and the columns right to left by $\mathds{Z}$, and denote the cell in the $i$th row and $j$th column by $(i,j)$ and its state in time period $t$ by $s_{i,j}(t)$.

As shown in Figure~\ref{fig:basefourneighbor}, the neighborhood of cell $(i,j)$ (white) is the set of cells $\{(i,j-1),(i-1,j),(i-1,j-1)\}$ (dark gray).

\begin{figure}[h]
  \centering
      \includegraphics[height=0.5in]{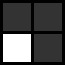}
  \caption{Neighborhood of CA2}
  \label{fig:basefourneighbor}
\end{figure}

If input $n$ has quaternary representation $\sum^{N}_{i=0}4^id_i$ and the smallest $i$ for which $d_i\neq 0$ is $L$, then the CA is initialized as follows: for an arbitrary $k$ and $i$ from $L$ to $N$, we have $s_{0,k+i}(0)=d_i$ and $s_{x,y}(0)=\texttt{empty}$ for all $(x,y)\not\in\{(0,k),...,(0,k+N)\}$.

\begin{figure}[h]
  \centering
      \includegraphics[width=\textwidth]{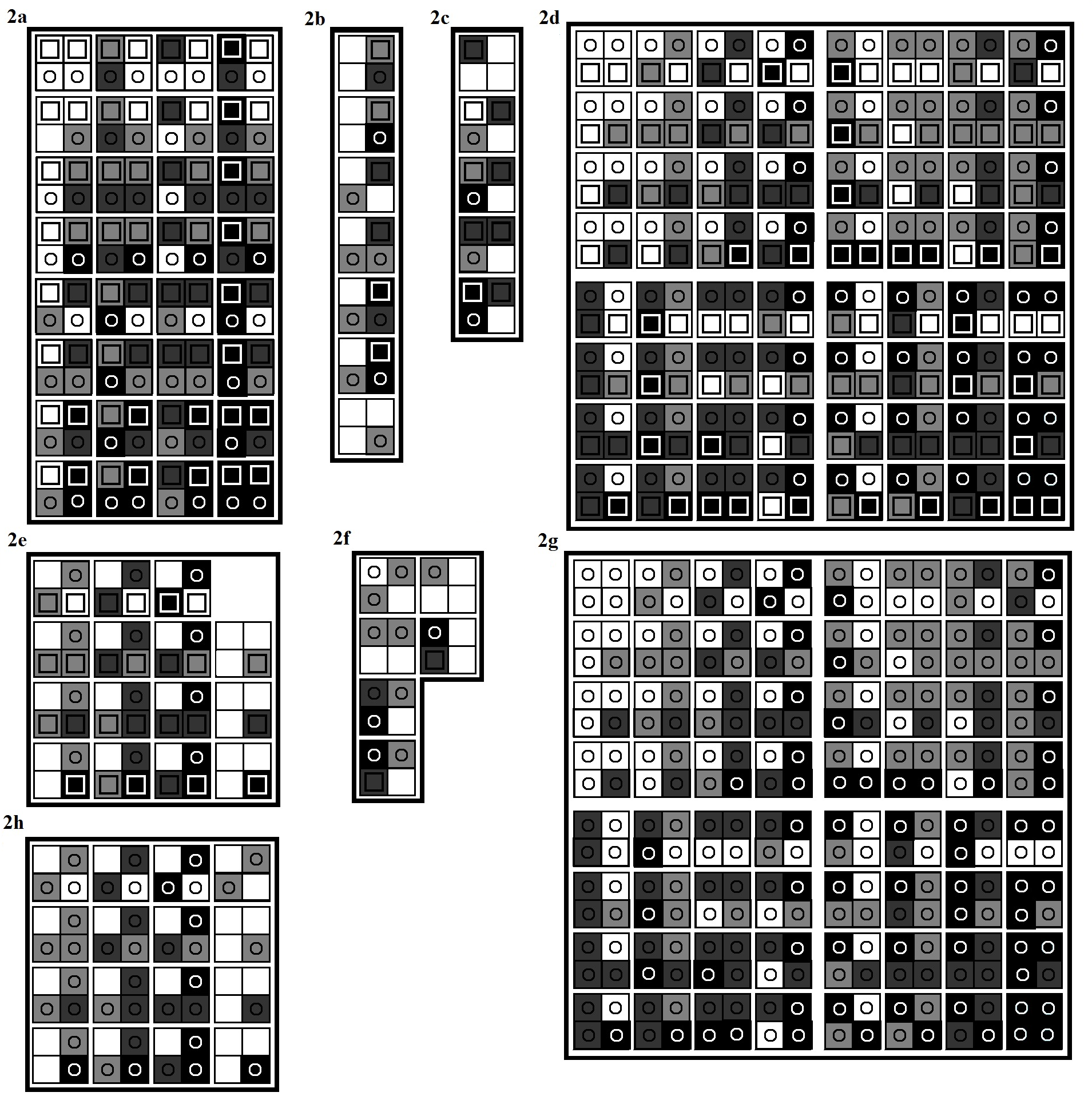}
  \caption{CA2 Evolution laws}
  \label{fig:basefourlaws}
\end{figure}

In Figure~\ref{fig:basefourlaws}, each $2\times2$ grid represents one evolution law, with each cell playing the same role as it does in Figure~\ref{fig:basefourneighbor}. In other words, the bottom-left cell in each $2\times2$ grid represents the cell that is transforming, and its color represents the future state into which it evolves given the states of its neighborhood. The colors of the remaining three cells represent the current states of the neighborhood that determine the state into which the bottom-left cell evolves. For example, the first of eight evolution laws in category 2b should be interpreted to mean that $s_{i,j}(t+1)=\texttt{empty}$ if $s_{i-1,j}(t)=\texttt{empty}$, $s_{i-1,j-1}(t)=\texttt{1-odd}$ and $s_{i,j-1}(t)=\texttt{3-even}$.

\begin{r1}If a cell's neighborhood corresponds to one of the evolution laws (Figure~\ref{fig:basefourlaws}), then that cell's state becomes the state of the cell as described in the evolution law.\end{r1}

\begin{r2}If a cell's neighborhood does not correspond to any evolution law, then that cell's state becomes \texttt{empty}.\end{r2}

\begin{figure}[h]
	\centering
	\includegraphics[height=1.2in]{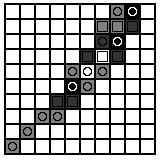}
	\caption{CA2 in action for $T^{0}=7$}
	\label{fig:ca2inaction}
\end{figure}

In Figure~\ref{fig:ca2inaction}, if we treat the non-empty cells in each row as quaternary digits, CA2 gives a sequence of iterates: $$7, 22, 11, 34, 17, 13, 5, 1, 1.$$ As we will prove, because CA2 mimics the following modified Collatz map $T_2:\mathds{N}\to\mathds{N}$ defined by \begin{equation*}T_2(n)=\begin{cases} (3n+1)/4^{k(3n+1)}, & n\equiv1\mod{2} \\ n/2, & n\equiv0\mod{2},\end{cases}\end{equation*} where $4^{k(x)}$ is the largest power of four dividing $x$, by eliminating trailing zeroes.

\begin{thm}
For every $m$, there exists some $i$ so that after $i$ iterations of CA2, the states of cells not in the state of \texttt{empty} in row $k$ correspond to the quaternary digits of $T^k_2(n)$, where $n$ is the input, for all $0\le k\le m$.
\end{thm}

\begin{proof}
Although parity in base four is given by the parity of the units digit, each digit must be represented by two states depending on the parity of the iterate for the CA to know whether to mimic $T(n)=n/2$ or $T(n)=3n+1$.

The former function is equivalent to $T(n)=2n/4$ and thus to multiplying by two and removing the resulting trailing zero; the latter function is equivalent to $T(n)=(4n+1)-n$ and thus, roughly speaking, to subtracting the iterate from a copy of itself shifted left one place and right-appended with a digit of 1.

We will denote the cell in the bottom-left of each evolution law in Figure~\ref{fig:basefourlaws}, that is, the cell that is evolving, by $(i,j)$, the current iterate by $\sum^{N'}_{k=0}4^kd'_k$, and the iterate in the preceding row by $\sum^{N}_{k=0}4^kd_k$.

We first prove that the rules in categories 2a to 2c are exhaustive in mimicking $T(n)=n/2$. Note again that multiplication by two involves carries of at most one.

If $s_{i,j-1}(t)=\texttt{empty}$ as in the rules in category 2c, the CA is determining the rightmost digits of the iterate on row $i$. By design, $d_0$ is necessarily nonzero, so for the iterate in row $i-1$ to be even, $d_0=2$, giving the first rule in category 2c. Here, $s_{i,j}(t+1)\equiv 2s_{i-1,j}(t)\pmod{4}$. There are then four choices-- \texttt{0-}, \texttt{1-}, \texttt{2-}, \texttt{3-even}-- for $d_1$, represented by $s_{i-1,j}(t)$ in the last four rules of category 2c, and because there is a carry of one from multiplying $d_0$ by 2, $d'_0=s_{i,j}(t+1)\equiv2s_{i-1,j}+1\pmod{4}$. 

If $s_{i-1,j}(t)=\texttt{empty}$ as in the rules in category 2b, the CA is determining the leftmost digits of the iterate on row $i$. If $s_{i-1,j-1}(t)\neq\texttt{empty}$, there are three possible states for cell $(i-1,j-1)$, namely \texttt{1-}, \texttt{2-}, \texttt{3-even}, and in each of these cases, there are $s_{i,j-1}=2s_{i-1,j-1}+\epsilon$, where $\epsilon\in\{0,1\}$, giving the first six evolution laws in category 2b. In these rules, the state into which cell $(i,j)$ evolves is chosen so that $4s_{i,j}+s_{i,j-1}=2s_{i-1,j-1}(t)+\epsilon$. If $d_{N'}=s_{i,j}=1$, then the cell to the left of this should be empty to mark the end of the iterate, giving the last rule in category 2b.

If none of the cells in the neighborhood of $(i,j)$ are in state \texttt{empty} as in the rules in category 2a, the CA is determining the inner digits of the iterate. Cells $(i-1,j)$ and $(i-1,j-1)$ can be in any of four states, namely \texttt{0-}, \texttt{1-}, \texttt{2-}, or \texttt{3-even}, and for each of these 16 cases, $s_{i,j-1}=2s_{i-1,j-1}+\epsilon$ where $\epsilon\in\{0,1\}$, giving the 32 evolution laws in category 2a. In these rules, the state into which cell $(i,j)$ evolves is chosen so that $$4s_{i,j}+s_{i,j-1}\equiv2\left(4s_{i-1,j}+s_{i-1,j-1}\right)+1\pmod{16}.$$

We next prove that the rules in 2d to 2h are exhaustive in mimicking $T(n)=3n+1$.

If $s_{i,j-1}(t)=\texttt{empty}$ as in the rules in category 2f, the CA is determining the rightmost digits of the iterate on row $i$. Because the iterate on row $i-1$ must be odd, there are two choices for $d_0$-- \texttt{1-odd} or \texttt{3-odd}-- represented by cell $(i,j-1)$ in the two rules of category 2f's right column (Figure~\ref{fig:basefourlaws}). If $d_0=\texttt{3-odd}$, $3d_0+1\equiv 2\pmod{4}$ so that $d'_0=s_{i,j}(t+1)=\texttt{2-even}$, but if $d_0=\texttt{1-odd}$, $3d_0+1\equiv 0\pmod{4}$ so that $s_{i,j}(t+1)=\texttt{empty}$ because CA2 deletes trailing zeroes. 

In the former case in which $d'_0$ has been determined, the CA then turns to evolution laws in category 2d to determine the inner digits of the iterate on row $i$. In the latter case in which $d'_0$ has not been determined, there are four possible values for $d_1$, namely \texttt{0-}, \texttt{1-}, \texttt{2-}, \texttt{3-odd}, giving the four evolution laws of category 2f's left column. By our interpretation of $3x+1$ as $(4x+1)-x$, the state into which cell $(i,j)$ evolves is chosen so that $d'_0=s_{i,j}(t+1)\equiv s_{i-1,j-1}(t)-s_{i,j}(t)\pmod{4}$ if $s_{i,j}(t+1)\neq\texttt{0-odd}$; otherwise, $s_{i,j}(t+1)=\texttt{empty}$ because CA2 deletes trailing zeroes. Proceeding as above, the CA turns to evolution laws in category 2d to determine the iterate's inner digits if $d'_0$ has been determined; otherwise, the above process is repeated until $d'_0$ is found.

If $s_{i-1,j}(t)=\texttt{empty}$ as in the rules in categories 2e and 2h, the CA is determining the leftmost digits of the iterate on row $i$. Observe that beyond the parity of this iterate, the rules in these categories are the same. 

In this case, if $s_{i-1,j-1}(t)\neq\texttt{empty}$, $(i-1,j-1)$ can be in any of three different states, namely \texttt{1-}, \texttt{2-}, \texttt{3-odd}, and $(i,j-1)$ can be any base-four digit in either parity. Certainly the state into which $(i,j)$ evolves, if not \texttt{empty}, must have the same parity as $(i,j-1)$ (e.g. if $s_{i,j-1}(t)=\texttt{3-even}$, $s_{i,j}(t+1)=\texttt{n-even}$ for some digit $n$). In particular, $s_{i,j}(t+1)\equiv s_{i-1,j-1}(t)-s_{i-1,j}(t)-\epsilon$, where $\epsilon$ equals 1 if $s_{i,j-1}+s_{i-1,j-1}\ge 4$ and 0 otherwise (see discussion of category 2d and 2g rules below).

On the other hand, if $s_{i-1,j-1}(t)=\texttt{empty}$, $(i-1,j)$ must be in state $\texttt{empty}$ as well, and $s_{i,j-1}(t)$ can be any nonzero digit in either parity. Here $s_{i,j-1}(t)$ must be $d'_{N'}$: otherwise, if we denote $T^k_2(n)$ to be the iterate in row $i-1$ and $y$ to be that in row $i$, we have that $y>T^{k+1}(n)$. Therefore, $s_{i,j}(t+1)=\texttt{empty}$.

If none of the cells in the neighborhood of $(i,j)$ are in state \texttt{empty} as in the rules in category 2d and 2g, the CA is determining the inner digits of the iterate on row $i$. Observe again that beyond the parity of this iterate, the rules in these categories are the same. Each cell in the neighborhood can be any base-four digit, and the state of $(i,j-1)$ can have either the \texttt{odd} or \texttt{even} attribute, giving 64 evolution laws for both category 2d and 2g. Because $3x+1=(4x+1)-x$, $s_{i,j-1}(t)\equiv s_{i-1,j-2}(t)-s_{i-1,j-1}(t)-\epsilon\pmod{4}$, where $\epsilon\in\{0,1\}$.

Certainly if $s_{i,j-1}(t)+s_{i-1,j-1}(t)\ge4$, the subtraction $s_{i-1,j-2}(t)-s_{i-1,j-1}-\epsilon$ was performed with a borrowing from $s_{i-1,j}(t)$, so in this case, $s_{i,j}(t+1)\equiv s_{i-1,j-1}(t)-s_{i-1,j}(t)-1\pmod{4}$. 

If $\epsilon = 1$ and $s_{i,j-1}(t)+s_{i-1,j-1}(t)=3$, then $s_{i-1,j-2}(t)=\texttt{0}$. Call this digit $d_k$ of the iterate on row $i-1$. Because $\epsilon = 1$, $d_{k-1}=\texttt{0}$ and there had to be a borrowing so that $d_{k-2}=\texttt{0}$. Continuing thus, we arrive at a contradiction, as the units digit of the iterate on row $i-1$ is nonzero.

So if $s_{i,j-1}(t)+s_{i-1,j-1}(t)<4$, the subtraction $s_{i-1,j-2}(t)-s_{i-1,j-1}-\epsilon$ was performed with no borrowing from $s_{i-1,j}(t)$, meaning $s_{i,j}(t+1)\equiv s_{i-1,j-1}(t)-s_{i-1,j}(t)\pmod{4}$.\end{proof}

\section{CA 3 -- BASE 2}

This cellular automaton, most efficient and most easily implementable CA, acts on an infinite two-dimensional grid of square cells, each of which can take on one of three states---represented in Figure~\ref{fig:basetwolaws} by the corresponding colors in parentheses--- \texttt{0} (dark gray), \texttt{1} (light gray), or \texttt{empty} (white).  The initial state of the CA is defined to be such that all cells are in the state of ``empty.''

Again, let the rows be indexed top to bottom by $\mathds{N}\cup\{0\}$ and the columns right to left by $\mathds{Z}$, and denote the cell in the $i$th row and $j$th column by $(i,j)$ and its state in time period $t$ by $s_{i,j}(t)$.

As shown in Figure~\ref{fig:basetwoneighbor}, the neighborhood of cell $(i,j)$ (white) is the set of cells $\{(i,j-1),(i-1,j),(i-1,j-1),(i-1,j-2)\}$ (dark gray).

\begin{figure}[h]
	\centering
			\includegraphics[height=0.5in]{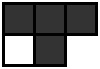}
	\caption{Neighborhood of CA3}
	\label{fig:basetwoneighbor}
\end{figure}

If input $n$ has binary representation $\sum^{N}_{i=0}2^id_i$ and the smallest $i$ for which $d_i\neq 0$ is $L$, then the CA is initialized as follows: for an arbitrary $k$ and $i$ from $L$ to $N$, we have $s_{0,k+i}(0)=d_i$ and $s_{x,y}(0)=\texttt{empty}$ for all $(x,y)\not\in\{(0,k),...,(0,k+N)\}$.

\begin{figure}[h]
  \centering
      \includegraphics[width=\textwidth]{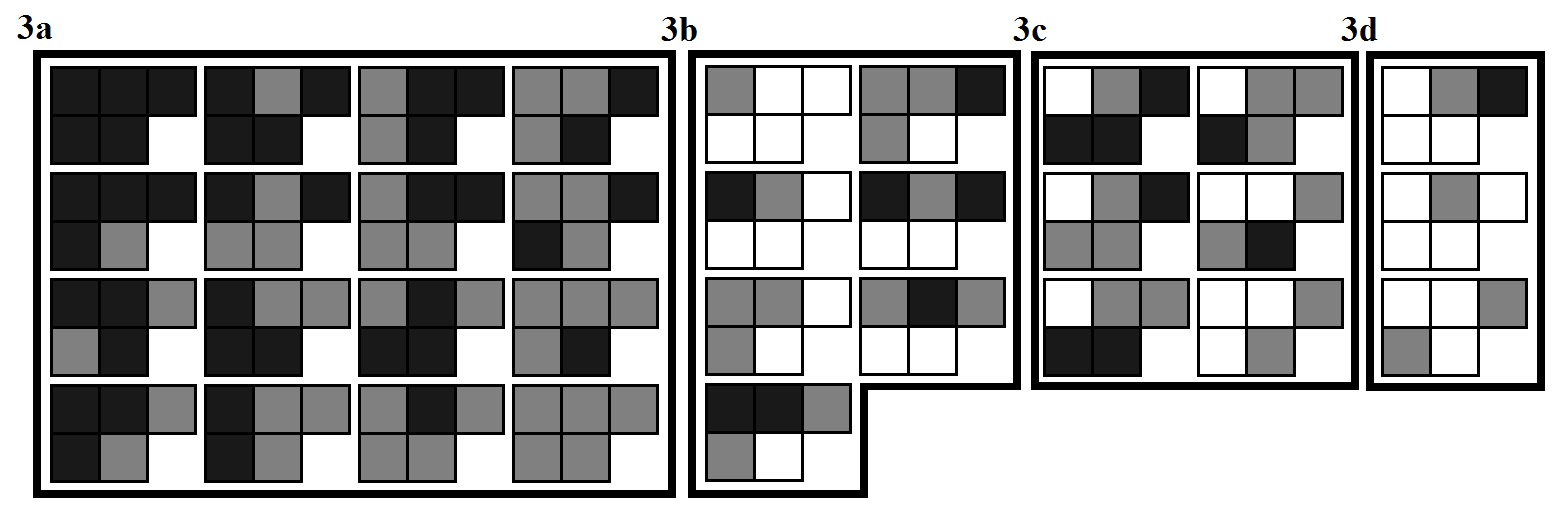}
  \caption{CA3 Evolution laws}
  \label{fig:basetwolaws}
\end{figure}

In Figure~\ref{fig:basetwolaws}, each group of five cells represents one evolution law, with each cell playing the same role as it does in Figure~\ref{fig:basetwoneighbor}. In other words, the bottom-left cell in each rule represents the cell that is transforming, and its color represents the future state into which it evolves given the states of its neighborhood. The colors of the remaining four cells represent the current states of the neighborhood that determine the state into which the bottom-left cell evolves. For example, the first of six evolution laws in category 3c should be interpreted to mean that $s_{i,j}(t+1)=\texttt{0}$ if $s_{i-1,j}(t)=\texttt{empty}$, $s_{i-1,j-1}(t)=\texttt{1}$ and $s_{i,j-1}(t)=\texttt{0}$, and $s_{i-1,j-2}(t)=\texttt{0}$.

\begin{r1}If a cell's neighborhood corresponds to one of the evolution laws (Figure~\ref{fig:basetwolaws}), then that cell's state becomes the state of the cell as described in the evolution law.\end{r1}

\begin{r2}If a cell's neighborhood does not correspond to any evolution law, then that cell's state becomes \texttt{empty}.\end{r2}

\begin{figure}[h]
	\centering
	\includegraphics[height=0.8in]{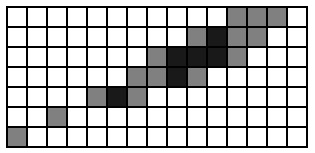}
	\caption{CA3 in action for $T^{0}=7$}
	\label{fig:ca3inaction}
\end{figure}

In Figure~\ref{fig:ca3inaction}, if we treat the non-empty cells in each row as quaternary digits, CA3 gives a sequence of iterates: $$7, 11, 17, 13, 5, 1, 1.$$ As we will prove, because CA3 mimics the following modified Collatz map $T_1:\mathds{N}\to\mathds{N}$ defined by \begin{equation*}T_3(n)=\begin{cases} (3n+1)/2^{k(3n+1)}, & n\equiv1\mod{2} \\ n/2, & n\equiv0\mod{2},\end{cases}\end{equation*} where $2^{k(x)}$ is the largest power of four dividing $x$, by eliminating trailing zeroes.

\begin{thm}
For every $m$, there exists some $i$ so that after $i$ iterations of CA3, the states of cells not in the state of \texttt{empty} in row $k$ correspond to the bits of $T^k_3(n)$, where $n$ is the input, for all $0\le k\le m$.
\end{thm}

\begin{proof}Like in base four, parity in binary is given by the parity of the units digit, but unlike CA2, CA3 can mimic $T(n)=n/2$ simply by representing trailing zeroes in an iterate as cells in the \texttt{empty} state. What remains is for the CA to mimic $T(n)=3n+1$.

This function is equivalent to $T(n)=(2n+1)+n$ and thus, roughly speaking, to adding the iterate to a copy of itself shifted left one place and right-appended with a digit of 1.

We will denote the cell in the bottom-left of each evolution law in Figure~\ref{fig:basetwolaws}, that is, the cell that is evolving, by $(i,j)$, the current iterate by $\sum^{N'}_{k=0}2^kd'_k$, and the iterate in the preceding row by $\sum^{N}_{k=0}2^kd_k$.

We will prove that the rules in categories 3a to 3d are exhaustive in mimicking $T(n)=3n+1$. Again, multiplication by two involves carries of at most one.

If $s_{i,j-1}(t)=\texttt{empty}$ but $s_{i-1,j}(t)\neq\texttt{empty}$ as in the rules in category 3b, the CA is determining the rightmost digits of the iterate on row $i$. By design, $d_0$ must be 1, giving the first rule in category 2c. Here, $s_{i,j}(t+1)\equiv 3s_{i-1,j}(t)+1\pmod{2}$. There are then two choices-- \texttt{0}, \texttt{1}-- for $d_1$, represented by $s_{i-1,j}(t)$ in the next two rules in category 2c's left column. In these two cases, the state into which $(i,j)$ evolves is chosen so that $3(2d_1+d_0)+1=3(2s_{i-1,j}(t)+s_{i-1,j-1}(t))\equiv s_{i,j}(t+1)\pmod{2}$ if $s_{i,j}\neq\texttt{0}$ and, because CA3 treats trailing zeroes as empty cells, \texttt{empty} otherwise.

In the former case in which the position of $d'_0=s_{i,j}(t+1)$ has been determined, the CA then turns to evolution laws in category 3a to determine the inner digits of the iterate on row $i$. In the latter case in which $d'_0$ has not been determined, there are two possible values for $d_1$, namely \texttt{0}, \texttt{1}, giving the bottom two evolution laws in category 3b. By our interpretation of $3x+1$ as $(2x+1)+x$, the state into which cell $(i,j)$ evolves is chosen so that $$d'_0=s_{i,j}(t+1)\equiv s_{i-1,j-1}(t)+s_{i-1,j}(t)+\epsilon\pmod{4}$$ where $\epsilon = 1$ if $s_{i-1,j-1}(t)+s_{i-1,j-2}(t)>s_{i,j-1}$ and $\epsilon = 0$ otherwise (see discussion of category 3a rules below) , provided that $s_{i,j}(t+1)\neq\texttt{0}$; if not, $s_{i,j}(t+1)=\texttt{empty}$ because CA3 treats trailing zeroes as \texttt{empty}. 

Again, in the former case in which the position of $d'_0$ has been determined, the CA then turns to rules in category 3a to determine inner digits. In the latter case in which $d'_0$ has not been determined, there are two possible values for $d_2$, namely \texttt{0}, \texttt{1}, giving the top two evolution laws in category 3b's right column. The state into which $(i,j)$ evolves, either \texttt{empty} or \texttt{1}, is determined in the same way as above. One more, the CA turns to evolution laws in category 2d to determine the iterate's inner digits if $d'_0$ has been determined; otherwise, the above process is repeated until the position of $d'_0$ is found.

If $s_{i-1,j}(t)=\texttt{empty}$ as in the rules in categories 3c and 3d, the CA is determining the leftmost digits of the iterate on row $i$. 

In this case, if $s_{i-1,j-1}(t)\neq\texttt{empty}$, it represents $d_N$ and thus must be \texttt{1}. There are only two cases in which at least one of cells $(i-1,j-2)$ and $(i,j-1)$ is \texttt{empty}: i) the iterate on row $i-1$ is of the form $x=\sum^{M}_{k=0}4^k$ so that $T(x)=4^{M+1}$ and the iterate on row $i$ is thus 1, ii) the iterate on row $i-1$ is $x=1$ so that the iterate on row $i$ is also 1. These cases give rise to the first two rules of category 3d, respectively, where $s_{i,j}(t+1)$ is \texttt{empty} rather than \texttt{0} because CA3 deletes trailing zeroes. The last rule of 3d then accounts for the case in which the iterate on row $i$ is 1.

Beyond these two cases, if neither cell $(i-1,j-2)$ nor cell $(i,j-1)$ is in the \texttt{empty} state, there are two possible states for each, giving the first four rules in category 3c. Here, the state into which $(i,j)$ evolves is again chosen so that $$s_{i,j}(t+1)\equiv s_{i-1,j-1}(t)+s_{i-1,j}(t)+\epsilon\equiv s_{i-1,j-1}(t)+\epsilon\pmod{4}$$ where $\epsilon = 1$ if $s_{i-1,j-1}(t)+s_{i-1,j-2}(t)>s_{i,j-1}$ and $\epsilon = 0$ otherwise.  

On the other hand, if $s_{i-1,j-1}(t)=\texttt{empty}$, $s_{i-1,j-2}(t)$ represents $d_N$. There are two possible non-empty states for cell $(i,j-2)$-- \texttt{0}, \texttt{1}-- giving the last two rules in category 3c. In the former case, $(i,j)$ must represent $d'_{N'}$; otherwise, if we denote this iterate to be $y$ and the preceding iterate to be $x$, $y>T_3(x)$, a contradiction. Therefore $s_{i,j}(t+1)=\texttt{1}$. In the latter case, $(i,j-1)$ must represent  $d'_{N'}$ for the same reason, so $s_{i,j}(t+1)=\texttt{0}$.

Finally, if none of the cells in the neighborhood of $(i,j)$ are in state \texttt{empty} as in the rules in category 3a, the CA is determining the inner digits of the iterate. Each of the cells in the neighborhood can be in any of two possible states, giving the sixteen rules in this category. Because $3x+1=(2x+1)-x$, $s_{i,j-1}(t)\equiv s_{i-1,j-2}(t)+s_{i-1,j-1}(t)+\epsilon\pmod{2}$, where $\epsilon\in\{0,1\}$. $$s_{i,j-1}(t)<s_{i-1,j-2}(t)+s_{i-1,j-1}(t)\le s_{i-1,j-2}(t)+s_{i-1,j-1}(t)+\epsilon$$ implies the addition $s_{i-1,j-2}(t)+s_{i-1,j-1}(t)$ was performed with a carry over to $s_{i-1,j}(t)$, so $s_{i,j}(t+1)\equiv s_{i-1,j-1}(t)+s_{i-1,j}(t)+1\pmod{2}$. On the other hand, $s_{i,j-1}(t)\ge s_{i-1,j-2}(t)+s_{i-1,j-1}(t)$ implies there was no carry, and $s_{i,j}(t+1)\equiv s_{i-1,j-1}(t)+s_{i-1,j}(t)\pmod{2}$.\end{proof}

\section{CONCLUSION}

Call the \emph{$n$-efficiency} of any of the above CAs the number of iterates that have been computed by the CA upon reaching 1 given input $n$, divided by the total stopping time of that input. The average $n$-efficiency over all $n$ such that $1\le n\le 2^{14}$ was determined to be roughly 69.4\%, 63.7\%, and 32.2\% for CA1, CA2, and CA3, respectively.

In addition, observe that it is not necessary to initialize any of these CAs with only one input. In fact, these CAs are able to compute the trajectories of an arbitrary number of inputs at a time. As an example, in the case of CA3, let the set of inputs be $\{n_j\}_{1\le j\le M}$, where input $n_j$ has binary representation $\sum^{N_j}_{i=0}2^id^j_i$ such that the smallest $i$ for which $d_i\neq 0$ is some $L_j$. The CA is initialized as follows: pick an arbitrary $k_0\in\mathds{Z}$. For all $j$ such that $1<j\le M$, let $k_j=k_{j-1}+N_{j-1}+\epsilon_j$, where $\epsilon_j$ is chosen so that the trajectories of $\sum^{N_j}_{i=0}2^id^j_i$ and $\sum^{N_{j-1}}_{i=0}2^id^{j-1}_i$ do not collide on the grid. Then for each $j$ such that $1\le j\le M$, let$s_{0,k_j+i}(0)=d^j_i$ for $0\le i\le N_j$. Let $s_{x,y}(0)=\texttt{empty}$ for all other pairs $(x,y)$.

\begin{figure}[h]
  \centering
      \includegraphics[width=0.65\textwidth]{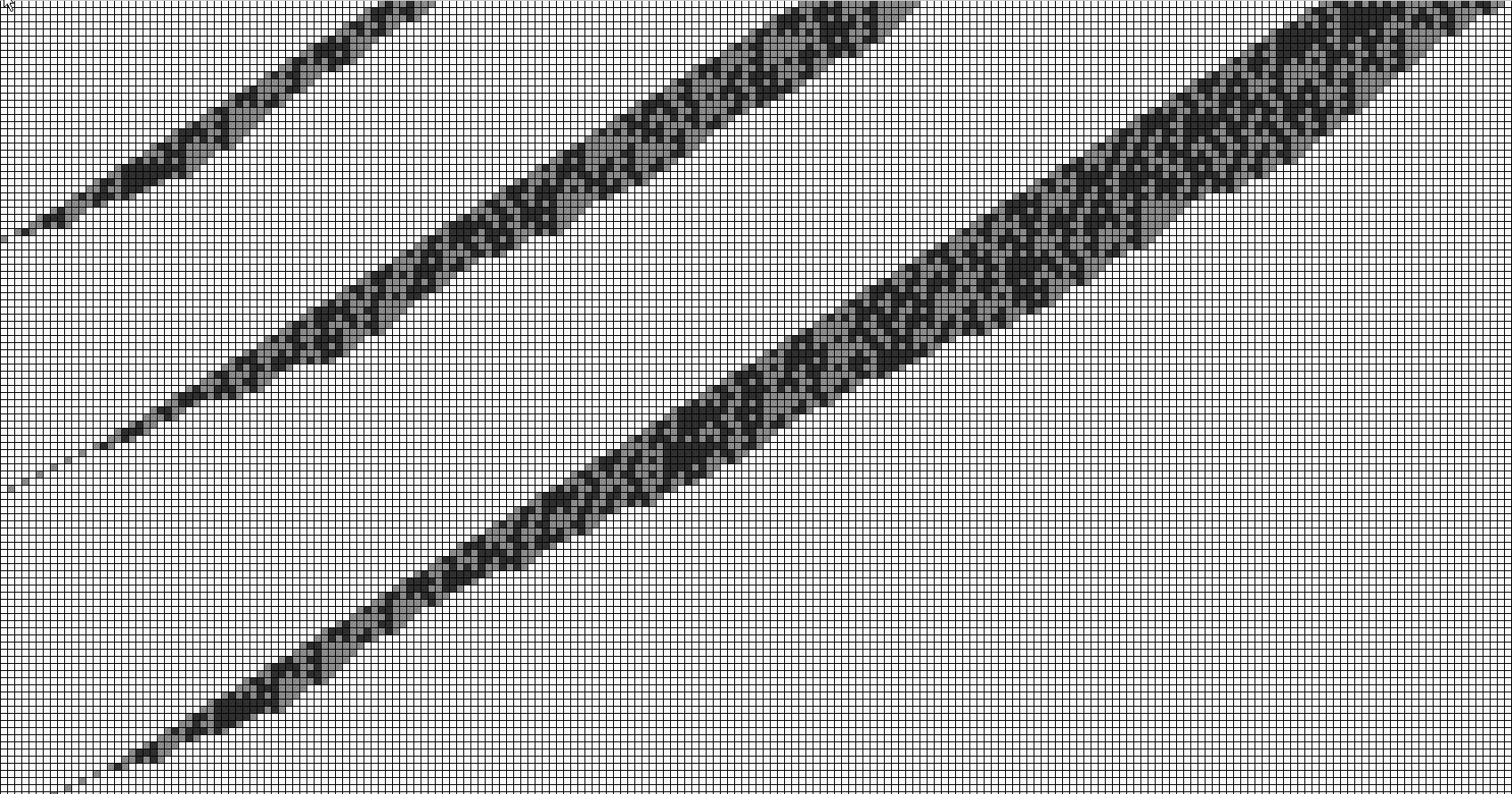}
  \caption{Parallel computing for initial iterates $n=$183, 120767, 53132499}
\end{figure}

Note however, that the difficulty in initializing a CA to compute multiple trajectories in parallel this way lies in the unpredictability of the behavior of $\epsilon_j$: it is not certain how large the space between two inputs must be for their trajectories not to collide.

As a workaround, if we place each iterate whose trajectory we want to compute on its own grid and, roughly speaking, stack those grids on top of one another to form an arbitrarily high three-dimensional grid of cubical cells, then this CA can verify multiple inputs in parallel.

\begin{figure}[h]
  \centering
      \includegraphics[height=1.5in]{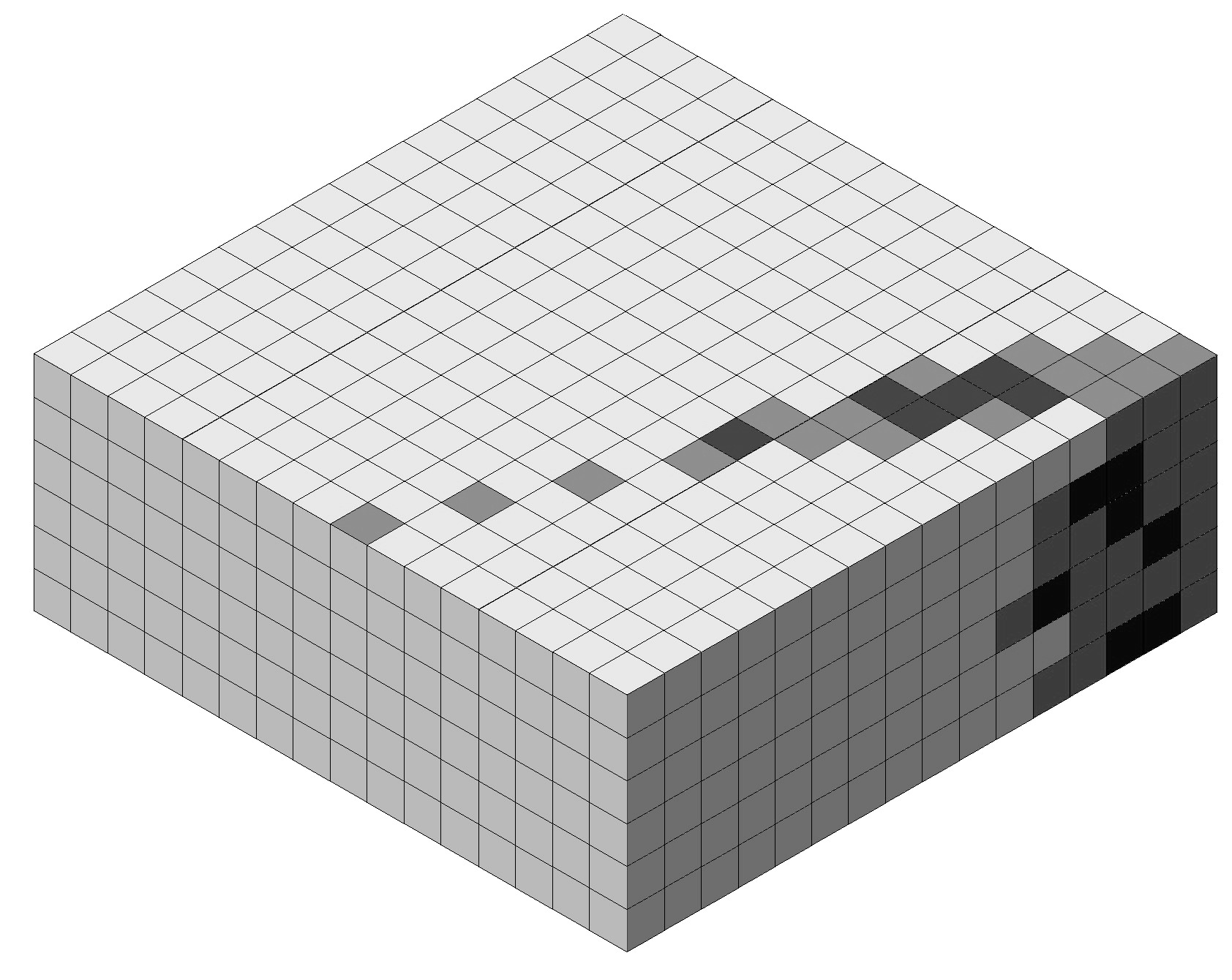}
  \caption{Parallel computing in three dimensions}
\end{figure}

In conclusion, we have found three CAs that mimic the behavior of the Collatz map in bases two, three, and four, and streamline the process of verification of the Collatz Conjecture by bypassing calculation of certain parts of the trajectory.  Beyond speeding the progress towards potentially finding a counterexample, this work affords insights into the distribution of the $\pmod{2^k}$ residues of iterates, which may justify the existing heuristic probabilistic argument that trajectories tend to decrease [7].

Topics for further study include applying the parallel computing model we propose to similar iterative computations, exploring whether it is possible to create evolution laws that depend entirely on digits from the previous iterate, and exploring other potentially viable radices for new CAs.  With regards to the latter, it is known that the distribution of $\left(3/2\right)^k\pmod{1}$ is closely related to the ``sorting properties'' of the Collatz map, that is, whether the function distributes odd iterates equally among the residues $\pmod{2^k}$ [3].  If no reasonable CA exists in base 3/2, it could still be worth exploring possible CAs for computing $\lfloor{\left(3/2\right)^k}\rfloor$. Implementations in Java of CA2 and CA3 can be found at \url{www.sitanchen.com/collatz4.html} and \url{www.sitanchen.com/collatz.html}, respectively.

\end{document}